\documentclass[11pt]{amsart}

\usepackage{amsmath,amssymb,amsthm}
\usepackage{graphicx}

\newtheorem{theorem}{Theorem}[section]
\newtheorem{lemma}[theorem]{Lemma}
\newtheorem{proposition}[theorem]{Proposition}

\theoremstyle{definition}
\newtheorem{definition}[theorem]{Definition}

\theoremstyle{remark}

\numberwithin{equation}{section}

\title{Existence of rotationally symmetric embedded f-minimal tori }

\begin{document}
	\author{Peng Peng}
	\address{School of Mathematical Sciences, Fudan University,
		Shanghai 200433, P.R. China} 
	\email{20110180011@fudan.edu.cn}

	\begin{abstract}
		Let $n\geq2$, and let $f:[0,+\infty)\to\mathbb{R}$ be a
		smooth function satisfying $f''\geq0$ and
		$0<m\leq f'\leq M$ for some constants $m$ and $M$.
		We prove that $\mathbb{R}^{n+1}$, endowed with the conformal
		metric
		\begin{equation*}
			ds^2=e^{-\frac{f(|x|^2)}{2n}}\sum_{i=1}^{n+1}dx_i^2,
		\end{equation*}
		admits a compact, embedded, rotationally symmetric minimal
		hypersurface diffeomorphic to $S^1\times S^{n-1}$.
		This extends Angenent's existence theorem for self-shrinking
		tori, which is recovered when $f(s)=s$.
	\end{abstract}
	\maketitle

\section{Introduction}
A one-parameter family of immersed hypersurfaces $M_t \subset \mathbb{R}^{n+1}$ evolves by mean curvature flow if
\begin{equation*}
	\frac{dx}{dt}= -H \mathbf{n},
\end{equation*}
where $x$ denotes the position vector, $H$ is the mean curvature, and $\mathbf{n}$ is the unit normal. Singularities may develop as the flow evolves. Self-shrinkers form an important class of singularity models and satisfy $H=\frac{1}{2} \langle x,\mathbf{n} \rangle.$

In 1986, Abresch and Langer \cite{Abresch1986TheNC} classified the closed self-shrinking curves. In 1992, Angenent \cite{Angenent1992} used a shooting method to construct a compact embedded self-shrinker with topology $S^{1}\times S^{n-1}$. In 2015, McGrath \cite{McGrath2015ClosedMC} used a related shooting method to construct self-shrinkers with topology $S^{n}\times S^{n} \times S^{1}$ in $\mathbb{R}^{2n+2}$. In 2023, Riedler \cite{RO2023} constructed a closed embedded self-shrinker with topology $S^{1}\times M$, where $M$ is an isoparametric hypersurface in $S^{n}$ whose principal curvatures have equal multiplicities. We refer the reader to \cite{Nguyen2006ConstructionOC,Nguyen2009Construction,Nguyen2014Construction,Moller2011ClosedSS,Kapouleas2011MeanCS} for further examples.

Self-shrinkers can be viewed as minimal hypersurfaces in $\mathbb{R}^{n+1}$ endowed with the Gaussian metric
\begin{equation*}
	e^{-\frac{|x|^{2}}{2n}} \sum_{i=1}^{n+1} dx_i^{2}.
\end{equation*}
As a natural generalization of self-shrinkers, $f$-minimal hypersurfaces have been studied in smooth metric measure spaces; see, for example, \cite{cheng2014eigenvalue,MR3324919}. Let $(\mathbb{R}^{n+1},ds^{2})$ be a Riemannian manifold with metric
\begin{equation}\label{eq1.1}
	ds^{2}=e^{-\frac{f}{2n}}\sum^{n+1}_{i=1}dx^{2}_{i},
\end{equation}
where $f=f(|x|^2)$ for a smooth function $f:[0,\infty)\to\mathbb{R}$. In this paper, an immersed hypersurface is called $f$-minimal if it is minimal with respect to \eqref{eq1.1}. Equivalently, it is a critical point of the weighted area functional $\int_\Sigma e^{-f(|x|^2)/4}d\mu$, where $d\mu$ is the Euclidean area element. The first variation gives
\begin{equation}\label{eq1.2}
	H=\frac{1}{4}\langle \nabla f, \mathbf{n} \rangle,
\end{equation}
where $\nabla f$ denotes the Euclidean gradient of the composite function $x\mapsto f(|x|^2)$. For geometric results on weighted minimal hypersurfaces under the hypotheses considered there, see \cite{cheng2014eigenvalue,MR3324919}.

The shooting method is a powerful tool for constructing special solutions to systems of ordinary differential equations (ODEs); see, for example, \cite{McGrath2015ClosedMC,john2017,Cheng2019ExamplesOC}. Under suitable assumptions on $f$, we use the shooting method to prove the existence of an embedded $f$-minimal torus. Our main result is the following theorem.
\begin{theorem}[Main Theorem]\label{th1}
	Let $n\geq2$, and let $f:[0,\infty)\to\mathbb{R}$ be smooth. Suppose that $f''\geq0$ and that $m\leq f'\leq M$ for constants $0<m\leq M$. Then $\mathbb{R}^{n+1}$, equipped with the metric \eqref{eq1.1}, admits a compact, embedded, rotationally symmetric $f$-minimal hypersurface diffeomorphic to $S^1\times S^{n-1}$.
\end{theorem}

Throughout the paper, we assume that $f$ satisfies the hypotheses of Theorem~\ref{th1}. The symbols $f'$ and $f''$ denote derivatives with respect to the scalar argument of $f$.

Let $\Gamma(t)=(x(t),r(t))$ be a curve in the upper half-plane, parametrized by Euclidean arclength. Rotating it about the $x$-axis generates a hypersurface $\Sigma$. For convenience, we view $\Gamma(t)=(x(t),r(t))$ either as a graph $r=u(x)$ over the $x$-axis or as a graph $x=g(r)$ over the $r$-axis.

A direct computation shows that the mean curvature of $\Sigma$ can be expressed as
\begin{equation*}
	H=\frac{n-1}{u\sqrt{1+(u')^{2}}}-\frac{u''}{(1+(u')^{2})^{\frac{3}{2}}}.
\end{equation*}
Equation \eqref{eq1.2} implies that the profile curve of a rotationally symmetric $f$-minimal hypersurface satisfies
\begin{equation}\label{xeq}
	\frac{d\theta}{dx}=\frac{u''}{1+(u')^{2}}=\frac{n-1}{u}+(xu'-u)\frac{f'}{2},
\end{equation}
where $f'=f'(x^2+u^{2})$ and $\theta$ denotes the angle that the profile curve $\Gamma$ makes with the $x$-axis.
When $\Gamma$ is instead represented as a graph $x=g(r)$ over the $r$-axis, the equation becomes
\begin{equation}\label{req}
	\frac{d\varphi}{dr}=\frac{g''}{1+(g')^{2}}=(r\frac{f'}{2}-\frac{n-1}{r})g'-\frac{f'}{2}g,
\end{equation}
where $f'=f'(r^{2}+g^{2})$ and $\varphi$ denotes the angle formed by the profile curve $\Gamma(t)=(x(t),r(t))$ and the $r$-axis. It is also useful to derive the corresponding system of ordinary differential equations for the arclength-parametrized curve $\Gamma (t)=(x(t),r(t))$:
\begin{equation}\label{arceq}
	\left.
	\begin{cases}
		\dot{x}=\cos\theta, \\
		\dot{r}=\sin\theta, \\
		\dot\theta=(\frac{n-1}{r}-\frac{r}{2}f')\cos\theta+\frac{x}{2}f'\sin\theta ,
	\end{cases}
	\right.
\end{equation}
where $f'=f'(x^{2}+r^{2})$, and a dot denotes differentiation with respect to the arclength parameter $t$. For each $R>0$, let $(x_R(t),r_R(t),\theta_R(t))$ be the maximal solution of \eqref{arceq} with initial conditions
\begin{equation}\label{arcdata}
	x(0)=0, \quad r(0)=R, \quad \theta(0)=0.
\end{equation}
We write $\Gamma_R(t)=(x_R(t),r_R(t))$ for the corresponding profile curve and use $x=g_R(r)$ and $r=u_R(x)$ for its graph representations on the relevant branches.

The main difficulty is that the profile curve need not be convex, which complicates the boundedness and convergence arguments.

We first show that, for sufficiently large $R$, the profile curve returns to the $r$-axis at a positive height before reaching $\theta_R=-\pi$ in the first quadrant. We then introduce a threshold $R^*$ and establish bounds for the corresponding curve segments. At the threshold, reflection across the $r$-axis produces a closed embedded profile curve. Rotating this curve about the $x$-axis yields the desired $f$-minimal hypersurface.

\section[Analysis of gR for Large R]{Analysis of $g_{R}$ for Large $R$}

When $f'\equiv1$, equations \eqref{xeq}, \eqref{req}, and \eqref{arceq} reduce to the corresponding rotational self-shrinker equations. Let $l=l(x)$ be the function defined implicitly by 
\begin{equation*}
	l^{2}=\frac{2(n-1)}{f'(x^{2}+l^{2})}.
\end{equation*}
It is easy to see that the graph of $r=l(x)$ must intersect the $r$-axis. We denote the point of intersection by $(0,\tilde{r})$. Recall that $u=\sqrt{2(n-1)}$ is a particular solution of the classical self-shrinker equation. The curve $r=l(x)$ plays a role analogous to that of the constant solution $u=\sqrt{2(n-1)}$ in the classical self-shrinker case. More precisely, if $u(x)<l(x)$, then
\begin{equation*}
	K_x(u):=\frac{n-1}{u}-\frac{u}{2}f'(x^{2}+u^{2})>0,
\end{equation*}
and	$K_x(u)\leq0$ if $u(x)\geq l(x)$. This follows from 
\begin{equation*}
	\frac{\partial K}{\partial u}=-\frac{n-1}{u^{2}}-\frac{f'}{2}-u^{2}f''<0,
\end{equation*}
together with the fact that $K(l)=0$. Thus, $u_R$ is decreasing for sufficiently small $x>0$ if $R>\tilde{r}$, and  increasing for sufficiently small $x>0$ if $R<\tilde{r}$. Moreover, if \eqref{xeq} admits a constant solution, it must be $r=\tilde{r}$. In this case, we have $f''\equiv0$ on the set $\{(x,r):x^2+r^2\geq \tilde{r}^2\}$.

It is straightforward to verify that there exists an $f$-minimal sphere. Indeed, if $R>0$ satisfies
$R^2=2n/f'(R^2),$ then $u=\sqrt{R^2-x^2}$ solves \eqref{xeq}. Since $m\leq f'\leq M$ and $f'$ is continuous, such a radius $R$ exists.

We begin with the following lemma, which collects some basic properties of $g$. Similar results have appeared in \cite{Kleene2010SelfshrinkersWA} and \cite{Drugan2013ImmersedS}.
\begin{lemma}\label{maximal}
	Suppose that $g$ is a nonnegative solution of  \eqref{req}. Then the following statements hold:
	\begin{enumerate}
		\item If $r_{0}$ is a critical point of $g(r)$, then $r_{0}$ is a local maximum.
		\item Suppose that $g\geq 0$ and $g'\geq 0$. If $g''(r_{1})<0$ at some $r_{1}$, then $g''(r)<0$ for every $r\in[0,r_{1}]$ in the domain of $g$.
		\item If a critical point of $u$ lies above the curve $r=l(x)$, then $u$ has a local maximum at that point. If a critical point of $u$ lies below the curve $r=l(x)$, then $u$ has a local minimum at that point.
	\end{enumerate}
\end{lemma}
\begin{proof}
	Part (1) follows immediately from \eqref{req} by setting $g'=0$. We prove part (2) by contradiction. Suppose that there is a point $r_{0}<r_{1}$ such that $g''(r_{0})=0$, with $g''>0$ immediately to the left of $r_{0}$ and $g''<0$ immediately to the right of $r_{0}$. Equation \eqref{req} implies that
	\begin{equation}\label{gsecond0}
		(rg'-g)(r_{0})=[\frac{2}{f'}\frac{n-1}{r}g'](r_{0})>0.
	\end{equation}
	Differentiating \eqref{req}, evaluating the result at $r_{0}$, and using \eqref{gsecond0}, we obtain
	\begin{equation*}
		g'''(r_{0})=(1+(g')^{2})[\frac{n-1}{r^{2}}g'+(r+gg')f''(rg'-g)](r_{0})>0.
	\end{equation*}
	This contradicts the assumption that $g''<0$ immediately to the right of $r_0$.
	
	Part (3) follows immediately from the facts that $K_x(u)>0$ whenever $u(x)<l(x)$ and $K_x(u)\leq0$ whenever $u(x)\geq l(x)$.
\end{proof}

\begin{lemma}\label{bigrestimate}
	Regard $\theta_R$ as a function of $r$. There exist an angle $\theta_1\in(-\pi/2,0)$ and a constant $C_{1}>0$, depending only on $n$, $m$, and $M$, such that, for all sufficiently large $R$, both $\theta_{R}(R-1/R)<\theta_1$ and $g_{R}(R-1/R)<C_{1}/R$ hold.
\end{lemma}
\begin{proof}
	Since $u'_R<0$ on $(0,g_R(R-1/R))$, equation \eqref{xeq} yields
	\begin{equation*}
		-\frac{\pi}{2}<\int_{0}^{x_1}\frac{d\theta_R}{dx}dx<\int_{0}^{x_1}\bigl[\frac{n-1}{u_{R}}-\frac{f'}{2}u_{R}\bigr]dx<\frac{2(n-1)-m(R-\frac{1}{R})^{2}}{2(R-\frac{1}{R})}x_1.
	\end{equation*}
	Setting $x_1=g_{R}(R-1/R)$, we obtain $g_{R}(R-1/R)< C_{1}/R$, where $C_{1}$ depends on $n$, $m$, and $M$.
	
	Equation \eqref{req} implies
	\begin{equation}\label{crucialineq1}
		\frac{d\varphi_{R}}{dr}\frac{\cos\varphi_{R}}{\sin\varphi_{R}}>(r\frac{m}{2}-\frac{n-1}{r}).
	\end{equation}
	Integrating this inequality from $R-1/R$ to $R$ and using the limiting value $\varphi_R(R)=-\pi/2$, we obtain
	\begin{equation*}
		\sin\varphi_{R}(R-\frac{1}{R})>-(\frac{R}{R-\frac{1}{R}})^{n-1}e^{\frac{m}{4}(-2+\frac{1}{R^2})}.
	\end{equation*}
	Since $\varphi_R+\theta_R=-\pi/2$ on this branch, the preceding estimate gives an angle $\theta_1\in(-\pi/2,0)$ such that $\theta_R(R-1/R)<\theta_1$ for all sufficiently large $R$.
\end{proof}

The following lemma is analogous to Lemma 2 in \cite{Angenent1992}.

\begin{lemma}\label{bigrbound}
	For sufficiently large $R$,  we have $g_{R}(r)\leq C_{2}/R$ for $r> \sqrt{8/m},$ where $C_{2}>0$ is a constant depending only on $n$, $m$, and $M$.
\end{lemma}
\begin{proof}
	Set $R_1:=R-1/R$. On the branch where $g_R'<0$, integration of \eqref{crucialineq1} from $r\in[\sqrt{8/m},R_1]$ to $R_1$ yields
	\begin{equation*}
		\begin{aligned}
			\sin\varphi_{R}(r)\geq(\frac{R_{1}}{r})^{n-1}e^{\frac{m}{4}(r^{2}-R_{1}^{2})}\sin\varphi_{R}(R_{1}).
		\end{aligned}
	\end{equation*}
	Since $\theta_R(r)\leq \theta_1$, we have
	\begin{equation*}
		g'_{R}(r)=\frac{\cos\theta_{R}(r)}{\sin\theta_{R}(r)}\geq \frac{\cos\theta_{R}(r)}{\sin\theta_R(R_1)}\geq \frac{\cos\theta_1}{\sin\theta_1}(\frac{R_{1}}{r})^{n-1}e^{\frac{m}{4}(r^{2}-R_{1}^{2})}.
	\end{equation*}
	Integrating the preceding inequality yields
	\begin{align}\label{grupperbound}
		g_{R}(r)&\leq g_{R}(R_{1})-\frac{\cos\theta_1}{\sin\theta_1}\int_{r}^{R_{1}}(\frac{R_{1}}{s})^{n-1}e^{\frac{m}{4}(s^{2}-R_{1}^{2})}ds.
	\end{align}
	For each fixed $n\geq2$, we estimate the integral in \eqref{grupperbound} for all sufficiently large $R$. Set $r_0:=\sqrt{8/m}$, and take $R$ sufficiently large that $R_1=R-1/R>2r_0$. For $r_0<r\leq R_1$, the nonnegativity of the integrand gives
	\begin{align*}
		I_R(r)
		&:=\int_r^{R_1} \left(\frac{R_1}{s}\right)^{n-1} e^{-\frac{m}{4}(R_1^2-s^2)}ds\\
		&\leq \int_{r_0}^{R_1/2} \left(\frac{R_1}{s}\right)^{n-1} e^{-\frac{m}{4}(R_1^2-s^2)}ds +
		\int_{R_1/2}^{R_1} \left(\frac{R_1}{s}\right)^{n-1} e^{-\frac{m}{4}(R_1^2-s^2)}ds\\
		&=:I_{R,1}+I_{R,2}.
	\end{align*}

	On $[r_0,R_1/2]$, we have $R_1^2-s^2\geq 3R_1^2/4$ and $(R_1/s)^{n-1}\leq(R_1/r_0)^{n-1}$. Hence
	\begin{equation*}
		I_{R,1}\leq \left(\frac{R_1}{r_0}\right)^{n-1}e^{-\frac{3m}{16}R_1^2}
		\left(\frac{R_1}{2}-r_0\right)\leq \frac{R_1^n}{2r_0^{n-1}}e^{-\frac{3m}{16}R_1^2}.
	\end{equation*}

	On $[R_1/2,R_1]$, we have $(R_1/s)^{n-1}\leq2^{n-1}$ and
	\begin{equation*}
		R_1^2-s^2=(R_1-s)(R_1+s)\geq R_1(R_1-s).
	\end{equation*}
	Consequently,
	\begin{equation*}
		I_{R,2} \leq 2^{n-1}\int_{R_1/2}^{R_1} e^{-\frac{mR_1}{4}(R_1-s)}ds=\frac{2^{n+1}}{mR_1}\left(1-e^{-\frac{m}{8}R_1^2}\right)\leq\frac{2^{n+1}}{mR_1}.
	\end{equation*}
	
	Combining these estimates, we obtain a constant $C_{3}>0$, depending only on $n$ and $m$, such that $I_R(r)\leq2C_{3}/R$ for all sufficiently large $R$ and all $r_0<r\leq R_1$. Together with \eqref{grupperbound} and Lemma~\ref{bigrestimate}, this gives $g_R(r)\leq C_{2}/R$ on the decreasing branch for $r_0<r\leq R_1$, where $C_{2}$ depends only on $n$, $m$, and $M$.
\end{proof}

\begin{lemma}\label{bigrvertical}
	For sufficiently large $R$, $g_{R}(r)$ has a critical point $r=\bar{R}$. Moreover, $g_{R}(\bar{R})\leq C_{2}/R$ and $\bar{R}\rightarrow +\infty$ as $R\rightarrow +\infty$.
\end{lemma}
\begin{proof}
	If $\bar{R}\geq \sqrt{8/m}$, then $g_{R}(\bar{R})\leq C_{2}/R$ follows from Lemma~\ref{bigrbound}. If $g_R$ has no critical point, then $g_R'<0$ on $(\sqrt{8/m},R)$. This case is included in the argument below.
	
	Assume, to the contrary, that there exist a constant $M_1>0$ and a sequence $R_k\to+\infty$ such that each $g_{R_k}$ either has no critical point or attains its maximum at a point $\bar{R}_k<M_1$. Define
	\begin{equation*}
		\bar g_k(r) := \frac{g_{R_k}(r)}{g_{R_k}(\bar M)},
		\quad \bar M := \max\left\{\sqrt{\frac{8}{m}}, M_1\right\}.
	\end{equation*}
	
	For brevity, write $g_k:=g_{R_k}$ and $\varphi_k:=\arctan g_k'$.

	Integrating \eqref{crucialineq1} from $\bar{M}$ to $r\leq R_k$ gives
	\begin{equation}\label{barm1}
		\sin\varphi_{k}(\bar{M})\geq(\frac{r}{\bar{M}})^{n-1}e^{\frac{m}{4}(\bar{M}^{2}-r^{2})}\sin\varphi_{k}(r).
	\end{equation}
	Since \eqref{barm1} and the endpoint limit $\sin\varphi_k(R_k)=-1$ imply $\sin\varphi_k(\bar M)\to0$, we have $\cos\varphi_k(\bar M)\geq1/2$ for all sufficiently large $k$, and hence
	\begin{equation}\label{barm2}
		g'_k(\bar{M})\geq 2\sin\varphi_k(\bar{M}).
	\end{equation}
	By \eqref{barm1}, for all sufficiently large $k$, we have
	\begin{equation*}
		g'_k<\sin\varphi_{k}(r)<(\frac{\bar{M}}{r})^{n-1}e^{\frac{m}{4}(r^{2}-\bar{M}^{2})}\sin\varphi_{k}(\bar{M}).
	\end{equation*}
	Thus, for all sufficiently large $k$,
	\begin{equation*}
		\begin{aligned}
			&g_k(R_k)-g_k(\bar{M})=\int_{\bar{M}}^{R_k}g'_k(t)dt<\sin\varphi_{k}(\bar{M})\int_{\bar{M}}^{R_k}(\frac{\bar{M}}{t})^{n-1}e^{\frac{m}{4}(t^{2}-\bar{M}^{2})}dt\\
			&<\sin\varphi_{k}(\bar{M})\int_{\frac{R_k}{2}}^{R_k}(\frac{\bar{M}}{t})^{n-1}e^{\frac{m}{4}(t^{2}-\bar{M}^{2})}dt+\sin\varphi_{k}(\bar{M})C_4R_k\\
			&<R_k^2\sin\varphi_{k}(\bar{M})+\sin\varphi_{k}(\bar{M})C_{4}R_k.
		\end{aligned}		
	\end{equation*}
	Here $C_{4}>0$ is the constant used to estimate the first part of the integral. The preceding estimate gives
	\begin{equation*}
		g_{k}(\bar{M})\geq-R_k^2\sin\varphi_{k}(\bar{M})-\sin\varphi_{k}(\bar{M})C_{4}R_k.
	\end{equation*}
	Combining this with \eqref{barm2} yields $0>\bar{g}'_k(\bar{M})>-C_{5}/R_k$ as $R_k\to +\infty$. In fact, by the same argument, we can prove $0>\bar{g}'_k(r)>-C_{5}/R_k$ for $r\in[\bar{M}, \bar M+1]$.
	
	It follows that $\bar{g}_{k}'\to 0$ uniformly on $[\bar{M}, \bar M+1]$ as $k\rightarrow+\infty$. It is easy to check that $\bar{g}_{k}''$ and $\bar{g}_{k}'''$ are uniformly bounded. By the Arzel\`a--Ascoli theorem, after passing to a subsequence, $\bar{g}_{k}\to\bar{g}=1$ in $C^3([\bar{M}, \bar M+1])$.
	
	Passing to the limit in the equation for $\bar{g}_k$, we find that $\bar{g}$ satisfies
	\begin{equation}\label{linear1}
		\bar{g}''=(\frac{r}{2}f'-\frac{n-1}{r})\bar{g}'-\frac{f'(r^2)}{2}\bar{g}.
	\end{equation}
	However, the constant function $\bar{g}\equiv1$ does not satisfy \eqref{linear1} on $[\bar M, \bar M+1]$, a contradiction.
\end{proof}

The preceding lemma and the following lemma concern the large-radius and small-radius regimes, respectively. The terms involving $rf'/2$ and $(n-1)/r$ in \eqref{req} play decisive roles in the two cases, respectively; therefore, the proof method for one of these lemmas can be applied to the other. The following lemma can also be proved using the method of Lemma 3 in \cite{Drugan2013ImmersedS}.

\begin{lemma}\label{smallrvertical}
	Let $g_a$ be the solution of \eqref{req} satisfying $g_a(a)\geq0$ and $g_a'(a)>0$. Then, for sufficiently small $a$, there exists $\bar{a}>a$ such that $g'_a(\bar{a})=0$. Moreover, $\bar{a}\to 0$ as $a\to0$.
\end{lemma}
\begin{proof}
	We argue by contradiction. Suppose that, for arbitrarily small $a$, there exists $d>a$ such that $g'_a(r)>0$ on $(a,d)$. On $(d/2,d)$, define 
	\begin{equation*}
		\tilde{g}_a:=\frac{g_a}{g_a(d/2)}.
	\end{equation*}
	We will prove that, after passing to a subsequence, $\tilde{g}_a\to1$ in $C^3([d/2,d])$ as $a\to0$. For $n>2$, this yields a contradiction by the linearized-equation argument used in the proof of Lemma~\ref{bigrvertical}. Therefore, there exists $\bar{a}\in(a,d)$ such that $g'_a(\bar{a})=0$. Since $d$ can be chosen arbitrarily small, it follows that $\bar{a}\to0$ as $a\to0$. We divide the proof into two steps.
	
	\medskip
	$\textbf{Step 1.}$ Equation \eqref{req} gives
	\begin{equation}\label{sinrupper}
		\frac{d\varphi}{dr}\frac{\cos\varphi}{\sin\varphi}<\frac{Mr}{2}-\frac{n-1}{r}.
	\end{equation}
	For an integer $k>1$, integrating the above inequality from $a$ to $ka$ yields
	\begin{equation}\label{sinrupperk}
		\sin\varphi_a(ka)<e^{\frac{M}{4}(k^2-1)a^2}(\frac{1}{k})^{n-1}\sin\varphi_a(a).
	\end{equation}
	
	By \eqref{xeq},
	\begin{equation*}
		\varphi_{a}(a)>\int_{g_a(a)}^{g_a(ka)}\frac{d\theta_a}{dx}dx>(\frac{n-1}{ka}-\frac{kMa}{2})[g_a(ka)-g_a(a)].
	\end{equation*}
	Thus, for sufficiently small $a$, there is a constant $C_{6}>0$, depending only on $n$, such that
	\begin{equation}\label{gakaupper}
		g_a(ka)<C_{6}ka\varphi_{a}(a)+g_a(a).
	\end{equation}
	We first consider the case in which 
	\begin{equation}\label{limit1}
		\lim_{a\to0}\frac{g_a(a)}{a\sin\varphi_a(a)}=+\infty.
	\end{equation}

	Integrating \eqref{sinrupper} from $a$ to $r$ gives
	\begin{equation}\label{sinestimate}
		\sin\varphi_a(r)<e^{\frac{M}{4}(r^2-a^2)}(\frac{a}{r})^{n-1}\sin\varphi_a(a).
	\end{equation}
	Combining this estimate with \eqref{limit1}, we obtain $\tilde{g}_a\to1$ in $C^3([d/2,d])$ as $a\to 0$.
	
	Next, consider the case 
	\begin{equation*}
		\lim_{a\to0}\frac{g_a(a)}{a\sin\varphi_a(a)}\leq C_{7},
	\end{equation*}
	for some constant $C_{7}>0$.
	
	We claim that, for sufficiently small $a$, $g'_a(ka)>a\sin\varphi_a(a)$. Suppose, to the contrary, that there is a sequence $a\to0$ with corresponding points $s\in(a,ka)$ such that $g'_a(s)\leq a\sin\varphi_a(a)$. Then for $r\in[a,s]$, \eqref{req} and \eqref{gakaupper} give
	\begin{equation*}
		\frac{d\varphi_a}{dr}\frac{\cos\varphi_a}{\sin\varphi_a}>\frac{mr}{2}-\frac{n-1}{r}-\frac{MC_{6}k+C_{7}}{2}.
	\end{equation*}
	Integrating this inequality from $a$ to $s$ yields
	\begin{equation}\label{g1kacontra}
		\sin\varphi_a(s)
		>e^{\frac{m}{4}(s^2-a^2)}(\frac{a}{s})^{n-1}e^{-\frac{1}{2}(MC_{6}k+C_{7})(s-a)}\sin\varphi_a(a).
	\end{equation}
	Thus, for sufficiently small $a$, $\sin\varphi_a(s)>\frac12(1/k)^{n-1}\sin\varphi_a(a)$. This contradicts the inequality $g'_a(s)\leq a\sin\varphi_a(a)$.  Therefore, $g'_a(ka)>a\sin\varphi_a(a)$. Using this estimate and arguing as in the derivation of \eqref{g1kacontra}, we obtain \eqref{g1kacontra} with $s=ka$. Combining this estimate with \eqref{sinrupperk} gives
	\begin{equation}\label{gakavalue}
		\frac{\sin\varphi_a(ka)}{\sin\varphi_a(a)}\to (\frac{1}{k})^{n-1},\quad \text{as} \quad a\to0.
	\end{equation}
	The same limit holds for each fixed positive integer smaller than $k$. Taking $k=2$, we obtain, for sufficiently small $a$, a constant $C_{8}$ depending on $n$ such that $g'_a(2a)>C_{8}\sin\varphi_a(a)$.
	Since $g''<0$ on $[a,2a]$, we have
	\begin{equation}\label{ga2a}
		g_a(2a)=\int_{a}^{2a}g'_a(r)dr+g_a(a)>C_{8}a\sin\varphi_a(a).
	\end{equation}
	
	\medskip
	$\textbf{Step 2.}$ Note that equation \eqref{gakavalue} implies that $\cos\varphi_a(2a)$ is bounded away from $0$. For $r\in[d/2,d]$, by \eqref{sinestimate} and \eqref{ga2a},
	\begin{equation*}
		\tilde{g}'_a=\frac{g'_a}{g_a(\frac{d}{2})}<\frac{g'_a}{g_a(2a)}<\frac{1}{C_{8}\cos\varphi_a(2a)}e^{\frac{M}{4}(r^2-a^2)}\frac{1}{r}(\frac{a}{r})^{n-2}\to0,
	\end{equation*}
	as $a\to0$, for $n>2$. Furthermore, \eqref{ga2a} implies $\tilde{g}_a\to1$ in $C^3([d/2,d])$ as $a\to0$. It remains to consider the case $n=2$.
	
	Fix an integer $N>1$. By \eqref{gakavalue}, we can choose $a>0$ sufficiently small so that $Na<d/2$ and the following
	estimates hold:
	\begin{equation*}
		g_a(\frac{d}{2})>\sum_{k=2}^{N}ag'_a(ka)>\sum_{k=2}^{N}a\sin\varphi_{a}(ka)>\frac{1}{2}a\sin\varphi_{a}(a)\sum_{k=2}^{N}(\frac{1}{k}).
	\end{equation*}
	By \eqref{sinestimate}, we obtain that $\tilde{g}_a\to1$ in $C^3([d/2,d])$ as $a\to0$ and $N\to+\infty$.
\end{proof}

\begin{definition}
	For each $R$, if $\theta_R$ first reaches $-\pi$ at a time $t>0$ before the curve $\Gamma_{R}$ intersects the $x$-axis or the $r$-axis, and at that time $x_R(t)>0$ and $r_R(t)>0$, then set $t_1(R):=t$. Otherwise, define $t_1(R)$ to be the first positive time at which $x_R(t)=0$ or $r_R(t)=0$. If none of these events occurs, let $t_1(R)$ be the maximal time that the solution exists.
\end{definition}

We now prove that, for sufficiently large $R$, the graph of $g_R$ intersects the $r$-axis at a point whose $r$-coordinate is bounded below by a positive constant independent of $R$.

\begin{lemma}\label{bigrdelta}
	There exists a constant $\delta>0$, independent of $R$, such that, for sufficiently large $R$, $x_{R}(t_{1}(R))=0$ and $r_{R}(t_{1}(R))\geq \delta.$
\end{lemma}

\begin{proof}
	We first prove $r_R(t_1(R))\geq \delta$. Suppose, for contradiction, that no such $\delta>0$ exists. Then by Lemma~\ref{bigrvertical}, there exists a sequence $R_{k}\to +\infty$ such that the corresponding solutions $g_{R_{k}}$, defined on $(r_{R_{k}}(t_{1}(R_{k})), R_{k})$, satisfy $r_{R_{k}}(t_{1}(R_{k}))\to 0$. However, Lemma~\ref{bigrvertical} implies that $g_{R_{k}}$ attains its unique maximum at $\bar{R}_{k}$ with $\bar{R}_{k}\to +\infty$.  On the other hand, Lemma~\ref{smallrvertical} applied near $r=r_{R_{k}}(t_{1}(R_{k}))$ implies that $\bar{R}_{k}\to 0$. This contradiction proves that there exists a constant $\delta > 0$ such that $r_R(t_1(R)) \geq \delta.$
	
	We now prove that $x_{R_k}(t_1(R_k))=0$ for sufficiently large $R_k$. Suppose, to the contrary, that $x_{R_k}(t_1(R_k))>0$ for all $k$. Note that $\theta_{R_k}(t_1(R_k))=-\pi$. Equation \eqref{xeq}, together with
	$r_{R_k}(t_1(R_k))\geq \delta$, implies that the solutions $u_{R_k}(x)$ have a common positive interval of existence near $x=x_{R_k}(t_1(R_k))$. This contradicts the uniform smallness estimate for $g_{R_k}$ established in Lemma~\ref{bigrvertical}. This completes the proof.
\end{proof}

\begin{definition}
	Define $R^*:=\inf\left\{\tilde {R}: \forall R>\tilde{R}, x_R(t_1(R))=0  \right\}.$
\end{definition}
Let $R_s$ be the radius of the $f$-minimal sphere. Its profile reaches the $x$-axis with $x=R_s>0$, so the definition of $R^*$ gives $R^*\geq R_s$. Since $R_s^2f'(R_s^2)=2n$ and $\tilde r^2f'(\tilde r^2)=2(n-1)$, strict monotonicity of $s\mapsto sf'(s)$ gives $R_s>\tilde r$. Hence $R^*>\tilde r$. 

\section{Existence of a Symmetric Solution}

A related result was established in \cite{Drugan2013ImmersedS}.
\begin{lemma}\label{unboundedsolution}
	Let $u:[0,\xi)$ be a positive solution of \eqref{xeq}. If $\xi<+\infty$, then $u$ is bounded on $[0,\xi)$.
\end{lemma}
\begin{proof}
	Suppose, to the contrary, that $u$ is unbounded on $[0,\xi)$.
	
	Following the argument in \cite{Drugan2013ImmersedS}, we first show that $u$ is strictly increasing and strictly convex on a terminal interval. Fix $a\in(0,\xi)$. Since $u$ is bounded on every compact subinterval of $[0,\xi)$ and $\xi<+\infty$, the function $u(x)/x$ is unbounded above on $[a,\xi)$. Set $\Psi(x):=xu'(x)-u(x)$. Note that $(u(x)/x)'=\Psi(x)/x^2.$ The mean value theorem therefore gives a point $x_0\in(a,\xi)$ such that $\Psi(x_0)>0.$ By \eqref{xeq},
	\begin{equation*}
		\Psi'(x) =xu''(x)=x[\frac{n-1}{u(x)}+\frac{f'(x^2+u(x)^2)}{2}\Psi(x)]\bigl(1+u'(x)^2\bigr).
	\end{equation*}
	The right-hand side is positive whenever $\Psi>0$. Consequently, $\Psi$ cannot decrease to zero after $x_0$, and hence
	$\Psi(x)>0$ for $ x_0\leq x<\xi.$ It follows that $u'(x)=(u(x)+\Psi(x))/x>0$ and $u''(x)>0$ for $x_0\leq x<\xi.$	
	
	Since $u$ is unbounded, we conclude that $\lim_{x\to\xi}u(x)=+\infty.$ Let $r_*=u(x_0)$, and denote the inverse of $u|_{[x_0,\xi)}$ by $g:[r_*,+\infty)\to[x_0,\xi).$ We then have $g''<0$, $g'>0$, and $\frac{d\varphi}{dr}\to0$. Moreover, there is an $r_0>\max\{r_*,2\}$ such that $g(r)\geq \xi/2$ and $\frac{d\varphi}{dr}>-m\xi/8$ for $r>r_0$.
	Using \eqref{req}, we then obtain, for $r\geq r_0$,
	\begin{equation*}
		\left(\frac{r}{2}f'-\frac{n-1}{r}\right)g'>\frac{m}{2}\frac{\xi}{2} -\frac{m\xi}{8}=\frac{m\xi}{8}.
	\end{equation*}
	Thus, setting $B_0:=m\xi/(4M)>0,$ we have $g'(r)>\frac{B_0}{r}.$ Integrating this on $\{r:r\geq r_0\}$, we have
	\begin{equation*}
		g(r)-g(r_0)>B_0\log\frac{r}{r_0}.
	\end{equation*}
	Hence $g(r)\to+\infty$ as $r\to+\infty.$ This contradicts $g(r)\to\xi.$ Thus $u$ is bounded on $[0,\xi)$.
\end{proof}

Following the variation-of-constants argument of Kleene and M{\o}ller~\cite{Kleene2010SelfshrinkersWA}, we obtain the following integral identity for a positive solution $u$ and any $0<x<a$ in its domain.
\begin{equation}\label{u-sigma1}
	\begin{aligned}
		u
		&=\frac{u(a)}{a}x+(u(a)-au'(a))x\int_{x}^{a}\frac{1}{t^{2}}e^{-\frac{1}{2}\int_{t}^{a}z(1+(u'(z))^{2})f'dz}dt\\
		&+(n-1)x\int_{x}^{a}\frac{1}{t^{2}}\bigg\{\int_{t}^{a}s\frac{1+(u'(s))^2}{u(s)}e^{-\frac{1}{2}\int_{t}^{s}z(1+(u'(z))^{2})f'dz}ds\bigg\}dt.
	\end{aligned}
\end{equation}

We record some estimates associated with the integral representation, following the method of \cite{Kleene2010SelfshrinkersWA}. For each fixed $x>0$, the integral multiplying the bounded coefficient $u(a)-au'(a)$ satisfies
\begin{equation}\label{integral1}
	x\int_{x}^{a}\frac{1}{t^{2}}e^{-\frac{1}{2}\int_{t}^{a}z(1+(u'(z))^{2})f'dz}dt \to0\qquad\text{as }a\to+\infty.
\end{equation}
Denote the third term by $\mathcal{R}_a(x)$. If $u(s)\geq u(x)$ for all $s\in[x,a]$, then
\begin{equation}\label{integral2}
	\begin{aligned}
		\mathcal{R}_a(x)&\leq\frac{(n-1)x}{u(x)}\int_{x}^{a}\frac{1}{t^{2}}\bigg\{\int_{t}^{a}s(1+(u')^2)e^{-\frac{1}{2}\int_{t}^{s}z(1+(u')^{2})f'dz}ds\bigg\}dt\\
		&\leq\frac{2(n-1)}{mu(x)}.
	\end{aligned}
\end{equation}
If instead $u(s)\geq cs$ on $[x,a]$ for a fixed $c>0$, the
same argument gives
\begin{equation}\label{integral3}
	0\leq\mathcal R_a(x)
	\leq\frac{n-1}{mc}\left(\frac1x-\frac{x}{a^2}\right)
	\leq\frac{n-1}{mcx}.
\end{equation}
Indeed, the inner integral in \eqref{u-sigma1} is at most
$2/(mct)$, and integration of $t^{-3}$ gives the stated bound.
This estimate does not require $u$ to be nondecreasing.

The method of Kleene and M{\o}ller yields the following proposition, proved in the appendix.
\begin{proposition}\label{conicalend}
	Every unbounded positive solution of \eqref{xeq} on $[0,+\infty)$ is asymptotic to a ray $r=\sigma x$ with $\sigma\geq0$. Moreover, any two such solutions with the same asymptotic slope coincide.
\end{proposition}

To apply the continuous dependence of solutions of \eqref{arceq} on initial data, we give a uniform bound on $t_1(R)$.
\begin{lemma}\label{t1bound}
	Suppose that there exist positive constants $\delta_1$, $B_1$, and $X$ such that $\delta_1<r_R(t)<B_1$,  $x_R(t)<X$, and $x_R(t)$ has at most one critical point before $t_1(R)$. Then $t_1(R)$ is uniformly bounded.
\end{lemma}
\begin{proof}	
	By \eqref{arceq},
	\begin{equation*}
		\dot{\theta}_R\leq \frac{n-1}{\delta_1}+\frac{M}{2}(B_1+X):=K.
	\end{equation*}
	Since $x_R(0)=0$ and $x_R(t)$ first increases and can then only decrease, we have, for $0<T<t_1(R)$,
	\begin{equation*}
		\int_{0}^{T}|\dot{x}|dt\leq 2X.
	\end{equation*}
	Note that $\ddot{x}=-\dot{\theta}\sin \theta $ and $\dot{r}=\sin\theta$. Differentiating $\dot{x}^2+\dot{r}^2=1$ yields $\ddot{r}=\dot{\theta}\dot{x}$. Integrating by parts and
	using $\dot x_R^2+\dot r_R^2=1$, we obtain
	\begin{align*}
		T&=\int_{0}^{T}\dot x_R^2 dt+\int_{0}^{T}\dot r_R^2 dt\\
		&=\int_{0}^{T}\dot x_R^2 dt+r_R\dot{r}_R|^{T}_0-\int_{0}^{T}r_R\ddot{r}_Rdt\\
		&\leq 2X+2B_1+B_1K\int_{0}^{T}|\dot{x}_R|dt\\
		&\leq 2X+2B_1+2B_1KX.
	\end{align*}
	It follows that $t_1(R)\leq2X+2B_1+2B_1KX$. 
\end{proof}

Unlike in the settings considered in \cite{Angenent1992,Cheng2019ExamplesOC,john2017,CLW2024}, it is difficult to prove that $\Gamma_{R}|_{[0,t_1(R)]}$ is convex for $R\geq R^*$. Indeed, Lemma~\ref{maximal} provides convexity only on one branch of $\Gamma_{R}|_{[0,t_1(R)]}$.  Fortunately, Proposition~\ref{conicalend} and the following lemma ensure that $\Gamma_{R}$ still enjoys sufficiently good properties.

\begin{lemma}\label{rstar1}
	For $R\geq R^*$, the curve segment $\Gamma_{R}|_{[0,t_1(R)]}$ is embedded. 	Moreover, we have
	\begin{enumerate}
		\item $\sup \{ x_R(t) : 0 < t < t_1(R), R\geq R^{*}\} < +\infty.$
		\item $\sup \{ r_R(t) : 0 < t < t_1(R), R^*+1\geq R\geq R^{*}\} < +\infty.$
	\end{enumerate}
\end{lemma}
\begin{proof}
	Assuming that $x_R$ is bounded on the curve segment, we first prove that $\Gamma_R|_{[0,t_1(R)]}$ is embedded for $R\geq R^*$. Lemmas~\ref{bigrvertical} and \ref{maximal}(2) imply that, for sufficiently large $R$, $\Gamma_{R}|_{[0,t_1(R)]}$ is a convex curve segment. Since $x_R(t_1(R))=0$, Lemma~\ref{maximal}(1) implies that there exists a unique $T_R$ such that $\theta_R (T_R)=-\pi/2$. At $T_R$, \eqref{arceq} gives $\dot{\theta}_{R}(T_R)=-\frac{f'}{2}x_R(T_R)<0$. Consequently, for sufficiently large $R$, the curve segment $\Gamma_{R}|_{[0,t_1(R)]}$ can be decomposed into the graphs
	\begin{equation*}
		r = u_{R,\pm}(x), \quad x \in [0, \xi_R],
	\end{equation*}
	where $\xi_R=x_R(T_R)$ and $u_{R,+}\geq u_{R,-}$. By the convexity of $\Gamma_R|_{[0,t_1(R)]}$, the graphs of $u_{R,-}$ and $u_{R,+}$ do not intersect away from their common endpoint when $R$ is sufficiently large. We decrease $R$ and consider the case in which $\Gamma_R|_{[0,t_1(R)]}$ is not convex. If an intersection occurred for some $R$, then there would exist a value $R_0$ for which the two graphs were tangent. The inequalities $u''_{R_0,+}(0)<0$, $u''_{R_0,-}(0)>0$, and $u'_{R_0,-}(0)\geq0$ exclude a first contact at $x=0$. Thus such a tangency must occur at a point with $x>0$. In that case, uniqueness of solutions to the ODE system implies that $u_{R_0,-}$ and $u_{R_0,+}$ coincide, which is a contradiction. In addition, Lemma~\ref{unboundedsolution} implies that both $u_{R,+}$ and $u_{R,-}$ are bounded.
	
	By continuous dependence on $R$, the boundedness assumption on $\xi_R$ is used to preserve the ordering $u_{R,+}\geq u_{R,-}$.
	
	Proof of (1).
	Suppose, for contradiction, that there exists a sequence $\{R_k\}\subset (R^{*},+\infty)$ such that $\xi_k \to +\infty.$ Lemma~\ref{bigrvertical} implies that the sequence $\{R_k\}$ is bounded. Denote its accumulation point by $R'$. We claim that $\xi_{R'}=+\infty$. Indeed, if $\xi_{R'}$ is bounded, Lemmas~\ref{unboundedsolution} and \ref{t1bound}, together with continuous dependence of the ODE solutions on their initial data, imply that $\xi_k$ is uniformly bounded for all sufficiently large $k$. We consider the largest $R'$ with $\xi_{R'}=+\infty$ in the following. Consequently, for all $R>R'$, the boundedness, embeddedness, and the ordering $u_{R,+}\geq u_{R,-}$ are preserved.
	
	We choose a sequence $R_k>R'$ converging to $R'$. Therefore, each $\Gamma_{R_k}$ is bounded and hence embedded.  We denote the corresponding graph functions by $u_{k,\pm}$ and their respective limits by $u_{R',\pm}.$  By continuous dependence of solutions of the ODE on their initial data, we have both $u_{k,+}\to u_{R',+}$ and $u_{k,-}\to u_{R',-}$ in $C_{\mathrm{loc}}^2([0,+\infty))$.

	By Lemma~\ref{maximal}, we have $u''_{k,-}\geq 0$ and $u'_{k,-}\geq 0$. Suppose that the lower branches ${u_{k,-}}$ admit a uniformly bounded subsequence. Passing to this subsequence, we deduce from the convexity of the lower branches that $u_{R',-}=\tilde{r}$ and $u_{R',+}>\tilde{r}$. Lemma~\ref{maximal}(3) then implies that $u_{R',+}$ has no critical points, and hence $u'_{R',+}<0$. We also conclude that $f''(s)\equiv0$ for all $s\geq \tilde{r}^2$. Consequently, arguing as in Lemma~\ref{maximal}(2), we have $u''_{R',+}<0$. This implies that $u_{R',+}=\tilde{r}$, which is impossible. It therefore remains to consider the case in which ${u_{k,-}}$ admits no uniformly bounded subsequence. 
	
	Thus, Proposition~\ref{conicalend} implies that both $u_{k,+}$ and $u_{k,-}$ converge to solutions with conical ends, denoted by $u_{\sigma_1}$ and $u_{\sigma_2}$. Note that we have already ruled out the case when $\sigma_1=0$ or $\sigma_2=0$. We assert that $\sigma_1>\sigma_2>0$. The ordering $u_{\sigma_1}(x)\geq u_{\sigma_2}(x)$, divided by $x$ and followed by the limit $x\to+\infty$, gives $\sigma_1\geq\sigma_2$. If $\sigma_1=\sigma_2$, then Proposition~\ref{conicalend} implies that $u_{\sigma_1}\equiv u_{\sigma_2}$. This is impossible, since $u''_{k,+}(0)<0$ is bounded away from $0$ and $u''_{k,-}(0)\geq0$.

	We next derive a contradiction under the assumption $\sigma_1>\sigma_2>0$.
	
	Using the local $C^1$ convergence of the lower branches and their convexity, choose $X_0>0$ sufficiently large that, for all sufficiently large $k$, 
	$u_{k,-}(X_0)\geq\frac{\sigma_2X_0}{2}$ and $u'_{k,-}(X_0)\geq\frac{\sigma_2}{2}.$ 
	We then have, for all sufficiently large $k$, $u_{k,+}(x)\geq u_{k,-}(x)\geq\frac{\sigma_2x}{2}$ for $X_0\leq x\leq\xi_k$. Set $Q(x):=u(x)-xu'(x).$ We divide the proof into two steps.
	
	\medskip
	$\textbf{Step 1.}$ It follows from \eqref{xeq} and $u''_{k,-}\geq 0$ that, for any $x\in[\xi_k-1,\xi_k]$, we have
	\begin{equation}\label{upperbound}
		Q_{k,-}(x)<\frac{2(n-1)}{mu_{k,-}(\xi_k-1)}.
	\end{equation}
	Equation \eqref{xeq} also implies that
	\begin{equation*}
		\frac{\pi}{2}\geq\theta_{k,-}(\xi_k)-\theta_{k,-}(\xi_k-1)\geq \int_{\xi_k-1}^{\xi_k} \frac{f'}{2}(xu'_{k,-}-u_{k,-}) dx.
	\end{equation*}
	By the mean value theorem for integrals, there exists $\xi'_{k}\in [\xi_k-1,\xi_k]$ such that
	\begin{equation}\label{lowerbound}
		Q_{k,-}(\xi'_{k})\geq -\frac{\pi}{m}.
	\end{equation}
	Apply \eqref{u-sigma1} with $u=u_{k,-}$ and $a=\xi'_{k}$. By \eqref{upperbound}, \eqref{lowerbound} and \eqref{integral1}, the boundary term tends to zero as $k\to+\infty$ for each fixed $x\geq X_0$. Thus the finite-endpoint identity gives
	\begin{equation*}
		u_{k,-}(x) =\frac{u_{k,-}(\xi'_k)}{\xi'_k}x+\mathcal{R}_{k,-}(x)+o_k(1),
	\end{equation*}
	where $o_k(1)\to0$ as $k\to+\infty$ with $x$ fixed. Since the lower branch is nondecreasing, \eqref{integral2} and $u_{k,-}(x)\geq\sigma_2x/2$ imply
	\begin{equation*}
		0\leq\mathcal{R}_{k,-}(x)\leq\frac{4(n-1)}{m\sigma_2x}.
	\end{equation*}
	Taking lower and upper limits as $k\to+\infty$ in the identity divided by $x$, and then letting $x\to+\infty$, we obtain
	\begin{equation}\label{sigma2-1}
		\lim_{k\to+\infty}\frac{u_{k,-}(\xi'_{k})}{\xi'_{k}}=\sigma_2.
	\end{equation}
	
	Write the inverse of $r=u_{k,-}(x)$ as $x=g_k(r)$. We set 
	\begin{equation*}
		I_k:=\{r\in(u_{k,-}(\xi'_k),u_{k,-}(\xi_k)):d\varphi_k/dr<-mg_k/4\}.
	\end{equation*}
	Then the measure of $I_k$ satisfies
	\begin{equation}\label{ikmeasure}
		|I_k|<\frac{2\pi}{mg_k(u_{k,-}(\xi'_{k}))}.
	\end{equation}
	Write $I_k=\cup_i(a_{i,k},b_{i,k})$. Applying the argument of Lemma~\ref{unboundedsolution} on the set $(u_{k,-}(\xi'_k),u_{k,-}(\xi_k))\setminus I_k$, we obtain
	\begin{equation}\label{g1estimate}
		g'_k>\frac{B_k}{r}, \quad B_k:=\frac{m}{2M}g_k(u_{k,-}(\xi'_{k})).
	\end{equation}
	Integrating \eqref{g1estimate} gives
	\begin{align*}
		1&\geq g_k(r)-g_k(u_{k,-}(\xi'_{k}))>B_k\{\log\frac{r}{u_{k,-}(\xi'_{k})}-\sum_{i}\log\frac{b_{i,k}}{a_{i,k}}\}\\
		&> B_k[\log\frac{r}{u_{k,-}(\xi'_{k})}-\frac{2\pi}{mg_k(u_{k,-}(\xi'_{k}))}],
	\end{align*}
	where 
	\begin{equation*}
		\log y-\log x<y-x, \quad  y>x>1,
	\end{equation*}
	and \eqref{ikmeasure} are used in the last inequality. This shows that $u_{k,-}(x)/u_{k,-}(\xi'_k)\to1$ uniformly for $x\in[\xi'_k,\xi_k]$ as $k\to+\infty$. Together with \eqref{sigma2-1}, we obtain
	\begin{equation}\label{sigma2-2}
		\lim_{k\to+\infty}\frac{u_{k,-}(\xi_k)}{\xi_k}=\sigma_2.
	\end{equation}
	
	\medskip
	$\textbf{Step 2.}$ We denote the largest critical point of $u_{k,+}$ by $x=\bar{x}_k$. We first prove that there exists $x'_k\in[\bar{x}_k-1,\bar{x}_k]$ such that $Q_{k,+}(x'_k)$ is bounded uniformly in $k$.
	
	If there exists $x'_k\in[\bar{x}_k-1,\bar{x}_k]$ such that $u''_{k,+}(x'_k)=0$, then $Q_{k,+}(x'_k)\to 0$ at these points. It remains to consider the case in which $u''_{k,+}\leq0$ on $[\bar{x}_k-1,\bar{x}_k]$ for sufficiently large $k$. It follows from \eqref{xeq} that $Q_{k,+}(x)>0$ on $[\bar{x}_k-1,\bar{x}_k]$; otherwise, $u''_{k,+}>0$, contradicting $u''_{k,+}\leq0$. Equation \eqref{xeq} also gives
	\begin{equation*}
		-\frac{\pi}{2}\leq\theta_{k,+}(\bar{x}_k)-\theta_{k,+}(\bar{x}_k-1)=\int_{\bar{x}_k-1}^{\bar{x}_k} \frac{n-1}{u_{k,+}}+\frac{f'}{2}(xu'_{k,+}-u_{k,+}) dx.
	\end{equation*}
	Consequently, by the mean value theorem for integrals, we obtain that, for some $x'_k\in [\bar{x}_k-1,\bar{x}_k]$, $Q_{k,+}(x'_k)$ is uniformly bounded above. Utilizing \eqref{integral3} and arguing as for the lower branch yields
	\begin{equation}\label{sigma1}
		\lim_{k\to+\infty}\frac{u_{k,+}(x'_k)}{x'_k}=\sigma_1\leq \lim_{k\to+\infty}\frac{u_{k,+}(\bar{x}_k)}{\bar{x}_k}.
	\end{equation}
	
	We claim that the point $(\bar{x}_k, u_{k,+}(\bar{x}_k))$ approaches $(\xi_k, u_{k,-}(\xi_k))$. Equation \eqref{xeq} implies that, on $[\bar{x}_k,\xi_k]$, 
	\begin{equation*}
		\frac{d\theta_k}{dx}<\frac{n-1}{u_{k,+}}-\frac{mu_{k,+}}{2}\to-\infty,
	\end{equation*}
	which forces $\bar{x}_k-\xi_k\to 0$. Similarly, \eqref{req} gives $u_{k,+}(\bar{x}_k)-u_{k,-}(\xi_k)\to0$. Combining this with \eqref{sigma2-2} and \eqref{sigma1}, we obtain
	\begin{equation*}
		\sigma_1\leq \lim_{k\to+\infty}\frac{u_{k,+}(\bar{x}_k)}{\bar{x}_k}=\lim_{k\to+\infty}\frac{u_{k,-}(\xi_k)}{\xi_k}=\sigma_2,
	\end{equation*}
	which contradicts $\sigma_1>\sigma_2$. Part (1) is proved.
	
	Proof of (2). Part (2) follows from part (1) and Lemma~\ref{unboundedsolution}.
\end{proof}
\begin{lemma}\label{rstar2}
	$\inf \{ r_R(t_1(R)) : R\geq R^*\} > 0.$
\end{lemma}
\begin{proof}
	For sufficiently large $R$, the result follows from Lemma~\ref{bigrdelta}. Suppose, to the contrary, that the conclusion fails. Then there exists a bounded sequence $\{R_k\}\subset (R^{*},+\infty)$ such that the corresponding curves $\Gamma_{R_k}$ satisfy
	\begin{equation*}
		(x_{R_k}(t_{1}(R_k)), r_{R_k}(t_{1}(R_k)))\to (0,0).
	\end{equation*}
	We use the same notation $u_{k,\pm}$ as in the previous lemma. Since the sequence $\{R_k\}$ is bounded, there exists a constant $B_2>0$, independent of $k$, such that $\xi_k>B_2$.
	
	For the lower branch $u_{k,-}$, we have $u'_{k,-}\geq 0$ and $u''_{k,-}>0$ by Lemma~\ref{maximal}. It follows from Lemma~\ref{smallrvertical} that $u_{k,-}(\xi_k)\to 0$. However, this implies that $u_{k,-}\to 0$ and $u'_{k,-}\to 0$ on $[0,B_2]$, which is impossible in view of \eqref{xeq}.
\end{proof}

The preceding two lemmas show that the curve segments $\Gamma_{R}|_{[0,t_1(R)]}$ remain regular as $R \to R^{*}$. We can therefore prove the main theorem by a continuity argument.

\begin{theorem}\label{thm:shooting}
	For $R = R^{*}$, the corresponding solution $(x_{R^{*}}(t), r_{R^{*}}(t))$ of \eqref{arceq}, subject to the initial conditions \eqref{arcdata}, satisfies $x_{R^{*}}(t_1(R^{*}))=0$ and $\theta_{R^{*}}(t_1(R^{*}))=-\pi.$
\end{theorem}
\begin{proof}
	By Lemmas~\ref{rstar1} and \ref{rstar2}, the family $(x_{R}(t), r_{R}(t))$ converges uniformly, up to a subsequence, to $(x_{R^{*}}(t), r_{R^{*}}(t))$. Notice that $x_{R^{*}}(t_1(R^{*}))=0$ and $\theta_{R^{*}}(t_1(R^{*})) \geq -\pi$. We argue by contradiction. If $\theta_{R^{*}}(t_1(R^{*})) > -\pi$, then, by continuous dependence on the initial data,
	there exists $\epsilon>0$ such that, for $\tilde{R}\in[R^*-\epsilon,R^*]$, $\theta_{\tilde{R}}(t_1(\tilde{R})) > -\pi$, which contradicts the definition of $R^{*}$. Therefore, $\theta_{R^{*}}(t_1(R^{*})) = -\pi$.
\end{proof}
\begin{proof}[Proof of Theorem~\ref{th1}]
	By Theorem~\ref{thm:shooting}, the curve $\Gamma_{R^*}|_{[0,t_1(R^*)]}$ and its reflection about the $r$-axis
	form a closed embedded profile curve. Rotating this curve about the $x$-axis yields the hypersurface asserted in the main theorem.
\end{proof}

	\section*{Appendix: Solutions with Conical Ends}
	
	\begin{proof}[Proof of Proposition~\ref{conicalend}]
		Under the assumptions $0<m\leq f'\leq M$ and $f''\geq0$, following the arguments in Section 2 of \cite{Kleene2010SelfshrinkersWA}, we can prove the $f$-minimal analogues of Lemmas 1 and 2 in that section. More precisely, along a suitable sequence $a\to+\infty$, the first term in \eqref{u-sigma1} converges to $\sigma x$ and the second term vanishes. The integral remainder $\mathcal{R}_a(x)$ tends to zero as $x\to+\infty$. Thus, an unbounded solution can be written as
		\begin{equation*}
			u_{\sigma}=\sigma x
			+(n-1)x\int_{x}^{+\infty}\frac{1}{t^{2}}\bigg\{\int_{t}^{+\infty}s\frac{1+(u'(s))^2}{u(s)}e^{-\frac{1}{2}\int_{t}^{s}z(1+(u')^{2})f'dz}ds\bigg\}dt.
		\end{equation*}
		The analogues of Lemma 3 and Proposition 1 in \cite{Kleene2010SelfshrinkersWA} can be established by similar methods. 
		
		Here, we prove only uniqueness. For fixed $\sigma>0$, consider the closed subset
		\begin{equation*}
			Y_{\sigma,b}:=\left\{v\in C^1_0([b,\infty)):0\leq v(x)\leq\frac{n-1}{m\sigma x},\quad|v'(x)|\leq\frac{2(n-1)}{m\sigma x^2}\right\}.
		\end{equation*}
		The preceding asymptotic estimates imply that every solution with asymptotic slope $\sigma$, after subtraction of $\sigma x$, belongs to this set.
		
		Let $v_1,v_2\in Y_{\sigma,b}$ and set $u_i(x):=\sigma x+v_i(x)$. Taking $b$ sufficiently large, we have
		\begin{equation*}
			\sigma x\leq u_i(x)\leq(\sigma+1)x, \qquad \frac{\sigma}{2}\leq u_i'(x)\leq\sigma+1 \qquad(x\geq b,\ i=1,2). 
		\end{equation*}
		Write
		\begin{equation*}
			y_i(z):=z^2+u_i(z)^2, \qquad f'_i(z):=f'(y_i(z)). 
		\end{equation*}
		All supremum norms below are taken on $[b,\infty)$. The letter $C$ denotes positive constants depending only on $n$, $m$, $M$, and $\sigma$.
		
		In this notation, the integral operator is
		\begin{equation*}
			[T_\sigma v_i](x)=(n-1)x\int_x^\infty\frac{1}{t^2}\bigg\{\int_t^\infty\frac{s(1+u_i'(s)^2)}{u_i(s)}e^{-\frac12\int_t^s z(1+u_i'(z)^2)f'_i(z)dz}ds\bigg\}dt.
		\end{equation*}
		
		\medskip
		$\textbf{Step 1.}$ We first establish an integral estimate. Set $\epsilon:=(\sigma+1)\|u_2-u_1\|_\infty$. Since $u_i>0$ and $u_i'>0$, we have $y_i'(z)\geq2z$. Moreover,
		\begin{equation*}
			|y_2(z)-y_1(z)|=|u_2(z)^2-u_1(z)^2|\leq2(\sigma+1)z\|u_2-u_1\|_\infty=2z\epsilon.
		\end{equation*}
		For $i=1,2$,
		\begin{equation*}
			y_i(z+\epsilon)-y_i(z)\geq(z+\epsilon)^2-z^2\geq2z\epsilon.
		\end{equation*}
		It follows that
		\begin{equation*}
			y_1(z+\epsilon)\geq y_2(z),\qquad y_2(z+\epsilon)\geq y_1(z). 
		\end{equation*}
		Since $f'$ is nondecreasing, so are the functions $f'_i$, and
		\begin{equation*}
			|f'_2(z)-f'_1(z)|\leq\bigl[f'_1(z+\epsilon)-f'_1(z)\bigr]+\bigl[f'_2(z+\epsilon)-f'_2(z)\bigr].
		\end{equation*}
		For every $A>b$,
		\begin{align*}
			\int_b^A \bigl[f'_i(z+\epsilon)-f'_i(z)\bigr]dz
			&=\int_A^{A+\epsilon}f'_i(z)dz-\int_b^{b+\epsilon}f'_i(z)dz\\
			&\leq(M-m)\epsilon.
		\end{align*}
		Letting $A\to+\infty$, we obtain
		\begin{align*}
			\int_b^\infty|f'_2(z)-f'_1(z)|dz
			&\leq2(M-m)\epsilon\\
			&=2(\sigma+1)(M-m)\|u_2-u_1\|_\infty.
		\end{align*}

		\medskip
		$\textbf{Step 2.}$ We next estimate the difference of the integral operators.
		The bounds on $u_i$ and $u_i'$ imply
		\begin{align*}
			\left|\frac{1+u_2'(s)^2}{u_2(s)}-\frac{1+u_1'(s)^2}{u_1(s)}\right|\leq\frac{C}{s^2}\|u_2-u_1\|_\infty+\frac{C}{s}\|u_2'-u_1'\|_\infty.
		\end{align*}
		
		The difference of the exponents satisfies
		\begin{align*}
			&\left|\frac12\int_t^sz\bigl[(1+u_2'(z)^2)f'_2(z)-(1+u_1'(z)^2)f'_1(z)\bigr]dz\right|\\
			&\qquad\leq C(s^2-t^2)\|u_2'-u_1'\|_\infty+C\int_t^s z|f'_2(z)-f'_1(z)|dz.
		\end{align*}
		Using
		\begin{equation*}
			|e^{-A}-e^{-B}|\leq e^{-\min\{A,B\}}|A-B|\qquad(A,B\geq0),
		\end{equation*}
		together with
		\begin{equation*}
			\int_t^\infty e^{-\frac m4(s^2-t^2)}ds\leq\frac{2}{mt},\qquad\int_t^\infty(s^2-t^2)e^{-\frac m4(s^2-t^2)}ds\leq\frac{8}{m^2t},
		\end{equation*}
		we obtain
		\begin{align*}
			|T_\sigma v_2(x)-T_\sigma v_1(x)|\leq&\frac{C}{x^2}\|u_2-u_1\|_\infty+\frac{C}{x}\|u_2'-u_1'\|_\infty\\
			&+Cx\int_x^\infty\frac{1}{t^2}\int_t^\infty e^{-\frac m4(s^2-t^2)}\left\{\int_t^s z|f'_2(z)-f'_1(z)|dz\right\}dsdt.
		\end{align*}
		
		We retain the exponential kernel when estimating the last term.
		For $z\geq t$,
		\begin{equation*}
			\int_z^\infty e^{-\frac m4(s^2-t^2)}ds\leq\frac{2}{mz}e^{-\frac m4(z^2-t^2)}.
		\end{equation*}
		Applying Tonelli's theorem yields
		\begin{align*}
			&x\int_x^\infty\frac{1}{t^2}\int_t^\infty e^{-\frac m4(s^2-t^2)}\left\{\int_t^s z|f'_2(z)-f'_1(z)|dz\right\}dsdt\\
			&\qquad\leq\frac{2x}{m}\int_x^\infty|f'_2(z)-f'_1(z)|\left\{\int_x^z\frac{1}{t^2}e^{-\frac m4(z^2-t^2)}dt\right\}dz.
		\end{align*}
		
		For $x\leq t\leq z$, we have $t^{-2}\leq x^{-2}$ and $z^2-t^2\geq z(z-t)$. Consequently,
		\begin{equation*}
			x\int_x^z\frac{1}{t^2}e^{-\frac m4(z^2-t^2)}dt\leq\frac1x\int_x^z e^{-\frac{mz}{4}(z-t)}dt\leq\frac{4}{mx^2}.
		\end{equation*}
		Combining this with the total integral estimate for $|f'_2-f'_1|$, we conclude that
		\begin{align*}
			&x\int_x^\infty\frac{1}{t^2}\int_t^\infty e^{-\frac m4(s^2-t^2)}\left\{\int_t^s z|f'_2(z)-f'_1(z)|dz\right\}dsdt\\
			&\qquad\leq\frac{8}{m^2x^2}\int_x^\infty|f'_2(z)-f'_1(z)|dz\\
			&\qquad\leq\frac{C}{x^2}\|u_2-u_1\|_\infty.
		\end{align*}
		
		Therefore,
		\begin{equation*}
			|T_\sigma v_2(x)-T_\sigma v_1(x)|\leq\frac{C}{x^2}\|u_2-u_1\|_\infty+\frac{C}{x}\|u_2'-u_1'\|_\infty.
		\end{equation*}
		
		To estimate the derivative, we differentiate the outer integral:
		\begin{equation*}
			[T_\sigma v_i]'(x) =\frac{T_\sigma v_i(x)}{x} -\frac{n-1}{x} \int_x^\infty \frac{s(1+u_i'(s)^2)}{u_i(s)} e^{-\frac12\int_x^s z(1+u_i'(z)^2)f'_i(z)dz} ds.
		\end{equation*}
		A similar argument then implies, for $x\geq b\geq1$,
		\begin{equation*}
			|(T_\sigma v_2-T_\sigma v_1)'(x)|\leq\frac{C}{x}\|u_2-u_1\|_\infty+\frac{C}{x^2}\|u_2'-u_1'\|_\infty. 
		\end{equation*}
		Recalling that $u_2-u_1=v_2-v_1$, we obtain
		\begin{equation*}
			\|T_\sigma v_2-T_\sigma v_1\|_{C^1}\leq \frac{C}{b}\|v_2-v_1\|_{C^1}.
		\end{equation*}
		
		It is easy to check that
		\begin{equation*}
			0\leq T_\sigma v(x)\leq\frac{n-1}{m\sigma x},\qquad|(T_\sigma v)'(x)|\leq\frac{2(n-1)}{m\sigma x^2},
		\end{equation*}
		so $T_\sigma$ maps $Y_{\sigma,b}$ into itself. Thus, for all sufficiently large $b$, it is a strict contraction, with contraction factor $\alpha_b=C/b<1$.
		
		Finally, if $u_1$ and $u_2$ are two solutions with the same asymptotic slope $\sigma$, their remainders $v_i=u_i-\sigma x$ satisfy $v_i=T_\sigma v_i$. The contraction estimate forces $v_1=v_2$ on $[b,\infty)$. Uniqueness for the ordinary differential equation then implies that $u_1=u_2$ on their entire common domain.
	\end{proof}

\section*{Use of generative AI}
During the preparation of this work, the author used ChatGPT (OpenAI) to suggest an application of Tonelli's theorem in the integral estimate in the proof of Proposition~\ref{conicalend} in Appendix~A and to assist in checking and clarifying the integral estimates in the proof of Lemma~\ref{bigrvertical}. The author independently verified the suggested application of Tonelli's theorem and takes full responsibility for the mathematical content and the final manuscript.
\bibliographystyle{plain}
\bibliography{reference}

\end{document}